\documentclass[10pt]{amsart}

\usepackage{float}
\usepackage{amscd,amsmath,amssymb,amsfonts,amsthm, ascmac, bm}
\usepackage{enumerate, comment, mathtools} 

\usepackage{caption}
\usepackage{subcaption}
\usepackage{hyperref}
\usepackage{relsize}

\usepackage{booktabs}
\usepackage{a4wide}

\usepackage[mathscr]{euscript}

\usepackage{ragged2e}
\usepackage{ytableau,varwidth}
\usepackage{diagbox}
\usepackage{cleveref}

\usepackage{tikz-cd}
\usepackage{tikz}
\usetikzlibrary{snakes, 
	3d, matrix, decorations.pathreplacing,calc,decorations.pathmorphing, fit, patterns, trees, decorations.markings, arrows, positioning}
\usetikzlibrary{calc}

\usepackage{diagbox}
 
\numberwithin{equation}{section}
\allowdisplaybreaks[1]

\usepackage{colonequals} 

\newcommand\st{\operatorname{st}}

\allowdisplaybreaks
\allowbreak 

\DeclareMathOperator{\Cone}{Cone}

\DeclareMathOperator{\GL}{GL}

\newcommand{\C}{{\mathbb{C}}}
\newcommand{\R}{{\mathbb{R}}}

\newcommand{\Z}{{\mathbb{Z}}}

\newcommand{\B}{\mathscr{B}}

\newcommand{\op}{\textbf{Op.\,1}}
\newcommand{\oop}{\textbf{Op.\,2}}

\newcommand{\AOP}{\mathsf{AOP}}

\theoremstyle{plain}
\newtheorem{theorem}{Theorem}[section]
\newtheorem{lemma}[theorem]{Lemma}
\newtheorem{proposition}[theorem]{Proposition}
\newtheorem{corollary}[theorem]{Corollary}

\theoremstyle{definition}
\newtheorem{example}[theorem]{Example}
\newtheorem{definition}[theorem]{Definition}

\theoremstyle{remark}
\newtheorem{remark}[theorem]{Remark}

\AddToHook{env/lemma/begin}{\crefalias{theorem}{lemma}}
\AddToHook{env/proposition/begin}{\crefalias{theorem}{proposition}}
\AddToHook{env/corollary/begin}{\crefalias{theorem}{corollary}}
\AddToHook{env/example/begin}{\crefalias{theorem}{example}}
\AddToHook{env/definition/begin}{\crefalias{theorem}{definition}}
\AddToHook{env/conjecture/begin}{\crefalias{theorem}{conjecture}}
\AddToHook{env/question/begin}{\crefalias{theorem}{question}}
\AddToHook{env/problem/begin}{\crefalias{theorem}{problem}}
\AddToHook{env/remark/begin}{\crefalias{theorem}{remark}}

\crefname{lemma}{Lemma}{Lemmas}
\crefname{theorem}{Theorem}{Theorems}
\crefname{proposition}{Proposition}{Propositions}
\crefname{question}{Question}{Questions}
\crefname{definition}{Definition}{Definitions}
\crefname{conjecture}{Conjecture}{Conjectures}
\crefname{figure}{Figure}{Figures}
\crefname{corollary}{Corollary}{Corollaries} 
\crefformat{equation}{(#2#1#3)}
\Crefformat{equation}{Equation #2(#1)#3}
\crefname{table}{Table}{Tables}

\hypersetup{colorlinks}
\hypersetup{
	unicode=false,          
	pdftoolbar=true,        
	pdfmenubar=true,        
	pdffitwindow=false,     
	pdfstartview={FitH},    
	pdftitle={My title},    
	pdfauthor={Author},     
	pdfsubject={Subject},   
	pdfcreator={Creator},   
	pdfproducer={Producer}, 
	pdfkeywords={keyword1} {key2} {key3}, 
	pdfnewwindow=true,      
	colorlinks=true,       
	linkcolor=blue,          
	citecolor=red,        
	filecolor=violet,      
	urlcolor=violet       
}

\begin{document}

\author{Junho Jeong}
\address[Junho Jeong]{Department of Mathematics,
	Chungbuk National University,
	Cheongju 28644, Republic of Korea}
\email{junhojeong@chungbuk.ac.kr}

\author{Jang Soo Kim}
\address[Jang Soo Kim]{Department of Mathematics,
	Sungkyunkwan University, 
	Suwon 16419, Republic of Korea}
\email{jangsookim@skku.edu}

\author{Eunjeong Lee}
\address[Eunjeong Lee]{Department of Mathematics,
	Chungbuk National University,
	Cheongju 28644, Republic of Korea}
\email{eunjeong.lee@chungbuk.ac.kr}


\title[Bott manifolds of BS type and assemblies of ordered partitions]{Bott manifolds of Bott--Samelson type and \\ 
	assemblies of ordered partitions}

\date{\today}

\subjclass[2020]{Primary: 14M25, 14M15; Secondary: 05A05, 05E14, 05A19}

\keywords{Schubert varieties, toric varieties, Bott--Samelson--Demazure--Hansen varieties, Bott manifolds}

\begin{abstract} 
  A Bott manifold is a smooth projective toric variety having an iterated $\C P^1$-bundle structure. A certain family of Bott manifolds is used to understand the structure of Bott--Samelson varieties (or Bott--Samelson--Demazure--Hansen varieties), which provide desingularizations of Schubert varieties. Indeed, each Bott--Samelson variety is diffeomorphic to a Bott manifold. However, not all Bott manifolds originate from Bott--Samelson varieties. Those that do are specifically referred to as Bott manifolds of \emph{Bott--Samelson type}. In this paper, we provide a characterization of Bott manifolds of Bott--Samelson type by exploring their relationship with combinatorial objects called assemblies of ordered partitions. Using this relationship, we enumerate Bott manifolds of Bott--Samelson type and describe isomorphic Bott manifolds of Bott--Samelson type in terms of assemblies of ordered partitions.
\end{abstract}

\maketitle


\section{Introduction}

A \emph{Bott manifold} is a smooth projective toric variety that
arises as the total space of an iterated sequence of
\(\C P^1\)-bundles. This iterated sequence of Bott manifolds is called
a \emph{Bott tower}, a notion introduced by Grossberg and
Karshon~\cite{GK94Bott}. The original motivation of Bott towers stems
from their connection to another interesting family of smooth
projective varieties, called \emph{Bott--Samelson
  varieties}\footnote{The construction of the Bott--Samelson varieties
  is due to Bott and Samelson~\cite{BottSamelson58} in the framework
  of compact Lie groups, and was later developed in the
  algebro-geometric setting by Hansen~\cite{Hansen73} and
  Demazure~\cite{Demazure74}. For this reason, Bott--Samelson
  varieties are also referred to as Bott--Samelson--Demazure--Hansen
  varieties.}. These varieties provide desingularizations of Schubert
varieties. Although a Bott--Samelson variety is not toric in general,
Grossberg and Karshon~\cite{GK94Bott} showed that it is always
diffeomorphic to a Bott manifold. Moreover, Pasquier~\cite{Pasquier10}
proved that there exists a toric degeneration in which the generic
fiber is a Bott--Samelson variety and the special fiber is the
corresponding Bott manifold considered in~\cite{Pasquier10}. Because
of this relation, Bott manifolds provide an interesting connection
between the geometry, topology, combinatorics, and representation
theory (see, for instance,~\cite{GK94Bott, Pasquier10, CHJ22}).

We note that not all Bott manifolds arise from Bott--Samelson
varieties. Indeed, in each dimension \(m\), there exist infinitely
many Bott manifolds, whereas there are only finitely many
Bott--Samelson varieties. A Bott manifold coming from a Bott--Samelson
variety is called a Bott manifold of \emph{Bott--Samelson type} (or of
\emph{BS type}). This family includes, in particular, toric Schubert
varieties in the full flag variety
(cf.~\cite[Exercise~18.2.7]{AndersonFulton24}
and~\cite[Theorem~4.23]{LMP_directed_Dynkin}).

In this paper, we investigate Bott manifolds of BS type by associating
them with certain combinatorial objects called \emph{assemblies of
  ordered partitions}. In order to state our main theorem, we first
introduce some notation and terminology. See \Cref{section_AOP} for
more details.

Let \( [n]:=\{ 1,\dots,n\} \). Given a sequence
\(\bm{i} = (i_1,\dots,i_m) \in [n]^m\) of integers between \(1\) and
\(n\), we associate to it a Bott manifold \(\B_{\bm{i}}\) of BS type
by considering an iterated sequence of \(\C P^1\)-fibrations, where
the fibration structures are determined by the sequence~\(\bm{i}\)
(see Definition~\ref{def_Bott_manifold_of_BS_type} for a precise
description). Thus, the family of Bott manifolds of BS type can be
parameterized by the sequences in \([n]^m\).

An \emph{assembly of ordered partition} of \( [m] \) with bound
\( n \) is a sequence \( \sigma=(\sigma_1,\dots,\sigma_\ell) \) of
ordered partitions \( \sigma_a = (\sigma_a^1,\dots,\sigma_a^{r_a}) \)
such that \( \min(\sigma_1) < \cdots < \min(\sigma_\ell) \),
\( r_1+\dots+r_{\ell} + ({\ell}-1) \leq n \), and
\( \{\sigma_a^b:a\in [\ell], b\in [r_a]\} \) is a partition of
\( [m] \) (see Definition~\ref{def_assembly_of_ordered_partitions} for
more details). Let \(\AOP(n,m)\) denote the set of such assemblies of
ordered partitions. For an ordered partition
\(\eta = (\eta^1,\dots,\eta^r)\), we write
\(\overline{\eta} = (\eta^r,\dots,\eta^1)\). We define the equivalence
relation \(\sim\) on \( \AOP(n,m) \) as follows: for
\(\sigma = (\sigma_1,\dots,\sigma_\ell), \tau =
(\tau_1,\dots,\tau_{\ell'}) \in \AOP(n,m)\), we have
\(\sigma \sim \tau\) if \( \ell=\ell' \) and for each
\( a\in [\ell] \) either \(\tau_a = \sigma_a\) or
\(\tau_a = \overline{\sigma}_a\). We now state our first main result.

\begin{theorem}\label{thm_Bott_matrices}
  Let \(n\) and \(m\) be positive integers. 
  There is a bijection between the set of Bott manifolds of BS type
  and the set of equivalence classes of \(\AOP(n,m)\) under the relation \(\sim\):
\[
\{\B_{\bm i}: \bm i \in [n]^m\} 
\stackrel{1:1}{\longleftrightarrow} 
\AOP(n,m)/\sim. 
\] 
\end{theorem}

To obtain this theorem, we analyze the \emph{Bott matrices}
corresponding to Bott manifolds of BS type, which are lower triangular
matrices that encode the iterated \(\C P^1\)-bundle structures. By
classifying all such Bott matrices, we enumerate the Bott manifolds of
BS type.

\begin{corollary}\label{cor_enumeration}
  The generating function for the number \(b(n,m)\) of Bott manifolds
  of BS type determined by the sequences in \([n]^m\) is given by
\[
	\sum_{n,m\ge1}  \frac{b(n,m)}{m!} x^m y^{n}  
	= \frac{1}{y(1-y)} \left( \exp \left( \frac{y}{2}
	\left( \frac{1}{1-y(e^x-1)} +y(e^x-1) -1 \right)\right) -1\right).
\]
\end{corollary}

The set \(\AOP(n,m)\) of assemblies of ordered partitions can also be
used to distinguish Bott manifolds of BS type up to isomorphisms as
toric varieties. Let
\(\sigma = (\sigma_1,\dots,\sigma_\ell) \in \AOP(n,m)\). For a simple
transposition \(s_i = (i,i+1) \in \mathfrak{S}_m\), we say that
\(s_i\) is an \emph{admissible transposition} for~\(\sigma\) if \(i\)
and \(i+1\) are not neighbors in \(\sigma\). Here, \(c\) and \(d\) are
\emph{neighbors} in \(\sigma\) if they belong to the same block of
some ordered partition \(\sigma_a\), or if they appear consecutively
within the same block of some ordered partition~\(\sigma_a\) (see
Definition~\ref{def_neighbors} for the precise definition). We define
another equivalence relation \( \approx \) on \( \AOP(n,m) \) as
follows: for \(\sigma, \tau \in \AOP(n,m)\), we have
\(\sigma \approx \tau\) if there exist
\( \tilde{\sigma}, \tilde{\tau}\in \AOP(n,m) \) with
\(\tilde{\sigma} \sim \sigma\) and \(\tilde{\tau} \sim \tau\) such
that there is a sequence of admissible transpositions sending
\(\tilde{\sigma}\) to~\(\tilde{\tau}\). With this terminology, we
state our second main theorem.

\begin{theorem}\label{thm_isom}
  Let \(n\) and \(m\) be positive integers. There is a bijection
  between the set of isomorphism classes of Bott manifolds of BS type
  \textup{(}as toric varieties\textup{)} and the set of equivalence
  classes of~\(\AOP(n,m)\) under the relation \(\approx\):
\[
\{\B_{\bm i} : \bm i \in [n]^m \}/(\text{isom})
\stackrel{1:1}{\longleftrightarrow} \AOP(n,m)/\approx. 
\]
\end{theorem}

We say that a Bott manifold is \emph{decomposable} if it is isomorphic to a product of lower-dimensional Bott manifolds, and \emph{indecomposable} otherwise.
Using the fact that any indecomposable Bott manifold of BS type corresponds to an assembly consisting of a single ordered partition (see Proposition~\ref{prop_indecomposable}), 
we analyze the isomorphism classes of indecomposable Bott manifolds of BS type in terms of the corresponding sequences (see Theorem~\ref{cor_indecomposable}).

We note that sequences with pairwise distinct integers correspond to
\emph{toric} Schubert varieties. More precisely, if
\(\bm i = (i_1,\dots,i_m)\) consists of distinct integers, then the
Schubert variety \(X_{s_{i_1} \cdots s_{i_m}}\) in the full flag
variety is isomorphic to the Bott manifold \(\B_{\bm i}\) of BS type.
The classification of isomorphism classes of toric Schubert varieties
was studied in~\cite{LMP_directed_Dynkin}. Using \Cref{thm_isom}, we
give another proof of their result in type \(A\) (see
Corollary~\ref{cor_compare_with_directed_Dynkin}). 

The structure of this paper is as follows. In Section~\ref{section_Bott_manifolds_of_BS_type}, we recall the definition of Bott manifolds and describe their fans, with a focus on the special case of BS type. 
In Section~\ref{section_AOP}, we introduce assemblies of ordered partitions and consider their connection with  Bott manifolds of BS type. We use this correspondence to enumerate Bott manifolds of BS type, proving Theorem~\ref{thm_Bott_matrices} and Corollary~\ref{cor_enumeration} in Section~\ref{section_enumeration}. In Section~\ref{section_isomorphism}, we analyze isomorphism classes of Bott manifolds of BS type via admissible transpositions on assemblies, and prove Theorem~\ref{thm_isom}. 

\subsection*{Acknowledgments}
J.~Jeong and E.~Lee were supported by the National Research Foundation of Korea(NRF) grant funded by the Korea government(MSIT) (No.\ RS-2022-00165641, RS-2023-00239947) and by POSCO Science Fellowship of POSCO TJ Park Foundation. 
J.~S.~Kim was supported by the National Research Foundation of Korea (NRF) grant funded by the Korea government RS-2025-00557835.

\section{Bott manifolds of Bott--Samelson type}\label{section_Bott_manifolds_of_BS_type}

In this section, we recall Bott manifolds from~\cite{GK94Bott}, which are smooth projective toric varieties. Moreover, we consider a special family of Bott manifolds consisting of \emph{Bott manifolds of Bott--Samelson type} (or of \emph{BS type}). 
\begin{definition}[{\cite[\S 2.1]{GK94Bott}}]
	A \emph{Bott tower} \(\B_{\bullet}\) is an iterated \(\C P^1\)-bundle starting with a point: 
	\[
	\begin{tikzcd}[row sep = 0.5em]
		\B_m \rar \arrow[r, "\pi_m"]
			& \B_{m-1}\arrow[r, "\pi_{m-1}"]
			& \cdots \arrow[r, "\pi_2"]
			& \B_1 \arrow[r, "\pi_1"]
			& \B_0,   \\
		P(\underline{\C} \oplus \xi_m) \arrow[u, equal]
			& & 
			& \C P^1 \arrow[u, equal]
			& \{ \text{a point}\} \arrow[u, equal]
	\end{tikzcd}
	\]
	where each \(\B_j\) is the complex projectivization of the Whitney sum of a holomorphic line bundle~\(\xi_j\) and the trivial line bundle \(\underline{\C}\) over \(\B_{j-1}\). The total space \(\B_m\) is called a \emph{Bott manifold}. 	
\end{definition}

A Bott manifold \(\B_m\) is a smooth projective toric variety by the
above construction. Indeed, the induced projective bundles of a
  sum of holomorphic line bundles over a smooth projective toric
  variety is again a smooth projective toric variety
(cf.~\cite[\S7.3]{CLS11Toric}). By considering the first Chern classes
of the holomorphic line bundles \(\xi_j\) used to construct~\(\B_m\),
we obtain the ray generators of the fan of \(\B_m\). Because of the
fibration structure, the Picard group of \(\B_{j-1}\) is isomorphic to
the free abelian group of rank \(j-1\)
(cf.~\cite[Exercise~II.7.9]{Hartshorne}). Indeed, the Picard group of
\(\B_{j-1}\) is generated by the line bundles \(\gamma_{j-1,k}\) for
\(1 \leq k < j\). Here, \(\gamma_{j,k}\) is the pullback of the
tautological line bundle over \(\B_k\) via the composition
\(\pi_{k+1} \circ \pi_{k+2}\circ \cdots \circ \pi_{j} \colon \B_j \to \B_{k}\) for
\(k < j\). Therefore, there exist integers \(b_{j,k} \in \Z\) for
\(1 \leq k < j\) such that
\[
\xi_j = \bigotimes_{1 \leq k < j} \gamma_{j-1,k}^{\otimes b_{j,k}},
\]
where, \(\gamma_{j-1,k}^{\otimes (-1)}\) means the dual of \(\gamma_{j-1,k}\).
Hence, the set \(\{b_{j,k}\}_{1 \leq k < j \leq m}\) of integers
determines the Bott manifold \( \B_m \).

Let \([B_{j,k}]\) be the \( m\times m \) lower triangular matrix
given by \(B_{j,k} = b_{j,k}\) and \( B_{k,j}=0 \) for
\(1 \leq k < j \leq m\), and \(B_{j,j} = -1\) for \( j\in [m] \). We
call such a matrix a \emph{Bott matrix}. A Bott manifold determines a
Bott matrix and vice versa. Indeed, as explained
in~\cite[\S2.3]{GK94Bott}, there is a bijection between the set of
Bott manifolds of dimension \(m\) and the set of \(m \times m\) lower
triangular integer matrices with diagonal entries \( -1 \).

The fan \(\Sigma \subseteq \R^m\) of \(\B_m\) can be expressed in terms of
the integers \(\{b_{j,k}\}_{1 \leq k < j \leq m}\) that determine~\(\B_m\).
For the Bott manifold \(\B_m\) determined by
\(\{b_{j,k}\}_{1 \leq k < j \leq m}\), the primitive generators of the rays
in the fan \(\Sigma \subseteq \R^m\) of \(\B_m\) are the column vectors of
the following matrix
\begin{equation}\label{PreBottM}
          [\mathbf{e}_1 \ \mathbf{e}_2 \ \cdots \ \mathbf{e}_m \ 
          \mathbf{v}_1 \ \mathbf{v}_2 \ \cdots \ \mathbf{v}_m] :=
	\begin{bmatrix}
		1 & 0 & \cdots & & 0 & -1 & 0 & \cdots & & 0 \\
		0& 1 & \ddots & &  & b_{2,1}& -1 & \ddots & & \vdots \\
		\vdots& \ddots & 1 & & \vdots & b_{3,1} & b_{3,2} & -1 & & \\
		& & & \ddots &0 & \vdots & \vdots & \ddots & \ddots & 0\\
		0& &  \cdots & 0& 1 & b_{m,1} & b_{m,2} & \dots & b_{m,m-1} & -1 
	\end{bmatrix}.
\end{equation}
Note that \(\mathbf{e}_1,\dots,\mathbf{e}_m\) are the standard basis
vectors of \(\R^m\), and the matrix
\([\mathbf{v}_1 \ \mathbf{v}_2 \ \cdots \ \mathbf{v}_m]\) is the Bott
matrix associated to~\(\B_m\). For any set \(S\) of \( m \) ray
generators, \( S \) forms a maximal cone of \(\Sigma\) if and only if
\[
\{\mathbf{e}_i, \mathbf{v}_i\} \not\subseteq S \quad \text{ for all }i \in [m]. 
\]
Because of this description, we can see that the fan \(\Sigma\) is the
normal fan of a polytope that is combinatorially equivalent to the
\(m\)-dimensional cube \([0,1]^m\).

\begin{example}
For \(m = 2\), the primitive ray generators of the fan \(\Sigma\) of \(\B_2\) are the column vectors of the following matrix. 
\[
\begin{bmatrix}
	1 & 0 & -1 & 0 \\
	0 & 1 & b_{2,1} & -1
\end{bmatrix} = [\mathbf{e}_1 \ \mathbf{e}_2 \ \mathbf{v}_1 \ \mathbf{v}_2]. 
\]
There are four maximal cones 
\[
\Cone(\{\mathbf{e}_1, \mathbf{e}_2\}), 
\Cone(\{\mathbf{e}_1, \mathbf{v}_2\}), 
\Cone(\{\mathbf{v}_1,\mathbf{e}_2\}),
\Cone(\{\mathbf{v}_1,\mathbf{v}_2\}), 
\]
and we depict the fan of \(\B_2\) when \(b_{2,1} = -2\) in \Cref{figfanB2}.

\begin{figure}
		\centering
		\begin{tikzpicture}
			\fill[yellow!20] (0,0)--(2,0)--(2,2)--(0,2)--cycle;
			\fill[orange!20] (0,0)--(0,2)--(-2,2)--(-2,-2)--(-1,-2)--cycle;
			\fill[blue!15] (0,0)--(-1,-2)--(0,-2)--cycle;
			\fill[green!15] (0,0)--(0,-2)--(2,-2)--(2,0)--cycle;
			
			\draw[gray,->] (-2,0)--(2,0);
			\draw[gray, ->] (0,-2)--(0,2);
			\draw[very thick,->] (0,0)--(1,0) node[at end, above] {\(\mathbf{e}_1\)};
			\draw[very thick, ->] (0,0)--(0,1) node[at end, right] {\(\mathbf{e}_2\)};
			\draw[very thick, ->] (0,0)--(-1,-2) node[at end, below] {\(\mathbf{v}_1\)};
			\draw[very thick, ->] (0,0)--(0,-1) node[at end, right] {\(\mathbf{v}_2\)};
		\end{tikzpicture}
		\caption{The fan of a Bott manifold \(\B_2\) for \(b_{2,1}  = -2\).}
		\label{figfanB2}
\end{figure}
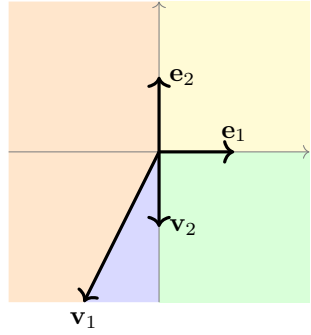
\end{example}

There is a special family of Bott manifolds called Bott manifolds of
BS type. Each of these is associated with a sequence of integers.

\begin{definition}\label{def_Bott_manifold_of_BS_type}
  Let \(\bm{i}= (i_1,\dots,i_m) \in [n]^m\). We define the Bott matrix
  \(B_{\bm{i}} = [b_{j,k}]\) as follows: For \(1 \leq k < j \leq m\),
  we set
  \[
    b_{j,k} = \begin{cases}
      -2 & \text{ if }|i_j - i_k| = 0, \\
      1 & \text{ if }|i_j - i_k| = 1, \\
      0 & \text{ otherwise},
    \end{cases}
    \qquad b_{k,j} = 0,
  \]
  and set \(b_{j,j} = -1\) for \(j \in [m]\). We denote by
  \(\B_{\bm i}\) the Bott manifold determined by the Bott matrix~\(B_{\bm i}\).

   We say that a Bott matrix \(B\) is \emph{of BS
    type} if \(B = B_{\bm i}\) for some \(\bm i\).
  We also say that a Bott manifold is \emph{of BS type} if the
  corresponding Bott matrix is of BS type. 
\end{definition}

\begin{example} 
  Let \(\bm i = (1,2,1,3)\). Then the Bott matrix \(B_{\bm i}\) is 
\[
B_{\bm i} = 
\begin{bmatrix}
	-1 & 0 & 0 & 0 \\
	1 & -1 & 0 & 0 \\
	-2 & 1 & -1 & 0 \\
	0 & 1 & 0 & -1
\end{bmatrix}. 
\]
\end{example}

\begin{definition}\label{def:1}
  For positive integers \( n \) and \( m \), we define \(B(n,m)\) to
  be the set of Bott matrices of BS type that can be obtained from the
  sequences in \([n]^m\):
\[
B(n,m) \colonequals \{ B_{\bm{i}} : \bm{i} \in [n]^m \}.
\]
We also define \(b(n,m)\) to be the number of elements in \(B(n,m)\).
\end{definition}

By the definition of \(B(n,m)\), we have
\(B(n_1,m) \subseteq B(n_2,m)\) whenever \(n_1 \le n_2\). Note that
every element in \( B(n,m) \) is an \( m\times m \) lower triangular
matrix such that the diagonal entries are all \( -1 \) and the entries
below the main diagonal are \(-2\), \( 1 \), or \(0\). Hence,
\( b(n,m)\le 3^{(m-1)m/2} \) for all \( n \), and the weakly
increasing sequence \((b(n,m))_{n \in \Z_{>0}}\) converges. The
following proposition shows that the sequence
\((b(n,m))_{n \in \Z_{>0}}\) has a constant value for \(n \geq 2m-1\).

\begin{proposition}\label{prop_seq_Bnm_converges}
  For any positive integers \( m \) and \( n \) with \(n \geq 2m-1\),
  we have \(B(n,m) = B(2m-1,m)\). 
\end{proposition}
\begin{proof}
  Let \(n \geq 2m-1\). We claim that \(B(n,m) = B(n+1,m)\). Since we
  have the inclusion \(B(n,m) \subseteq B(n+1,m)\), it suffices to
  prove that \(B(n+1,m) \subseteq B(n,m)\).

  Suppose \( B\in B(n+1,m) \). Then \( B=B_{\bm i} \) for some
  \(\bm i = (i_1,\dots,i_m) \in [n+1]^m\). If \( n+1 \) is not an
  element of \( \bm i \), then \( \bm i \in [n]^m \). Hence,
  \( B_{\bm i}\in B(n,m) \). If \( 1 \) is not an element of
  \( \bm i \), then \( B_{\bm i} = B_{\bm i'} \), where
  \( \bm i' = (i_1-1,\dots,i_m-1) \in [n]^m \). Thus, we also have
  \( B_{\bm i}\in B(n,m) \).

  Now suppose that \( \bm i \) contains both \( 1 \) and \( n+1 \).
  Let \( (a_1 , \dots, a_m) \) be the weakly increasing sequence
  obtained by sorting the elements in \( \bm i \). Since \( a_1=1 \),
  \( a_m= n+1 \), and
  \( \sum_{j=1}^{m-1} (a_{j+1}-a_{j}) = n \ge 2m-1 \), there exists
  \( j\in [m-1] \) such that \( a_{j+1}-a_{j}>2 \). Let
  \(\bm{i}' = (i_1',\dots,i_m')\in [n]^m \) be the sequence given by
\[
i_k' = \begin{cases}
	i_k & \text{ if } i_k\le a_j, \\
	i_{k} -1 & \text{ if } i_k \ge a_{j+1}. 
\end{cases}
\]
By \Cref{def_Bott_manifold_of_BS_type} and the inequality
\( a_{j+1}-a_{j}>2 \), we have \(B_{\bm i} = B_{\bm i'}\). Therefore,
\(B_{\bm i} \in B(n,m)\). Accordingly, the claim
\(B(n+1,m) \subseteq B(n,m)\) holds and we are done.
\end{proof}

\begin{remark}\label{rmk_BS_type} 
Let \(G = \GL_{n+1}(\C)\) and let \(B\) be a Borel subgroup of \(G\). 
For a sequence \(\bm i\), we have another geometric object, called a \emph{Bott--Samelson variety} \(Z_{\bm i}\), which is defined by the quotient space 
\[
Z_{\bm i} \colonequals (P_{i_1} \times P_{i_2} \times \cdots \times P_{i_m})/B^m. 
\]
Here, for \(i \in [n]\), \(P_i\) is the minimal parabolic subgroup of \(G\) containing \(B\),
and the action of~\(B^m\) on the product of minimal parabolic subgroups is defined by 
\[
(b_1,b_2,\dots,b_m) \cdot (p_1,p_2,\dots,p_m) 
= (p_1b_1,b_1^{-1}p_2b_2, \dots, b_{m-1}^{-1}p_mb_m) 
\]
for \((b_1,b_2,\dots,b_m) \in B^m\) and \((p_1,p_2,\dots,p_m) \in P_{i_1} \times P_{i_2} \times \cdots \times P_{i_m}\). 
A Bott--Samelson variety is not necessarily toric, but it is diffeomorphic to the Bott manifold \(\B_{\bm i}\) as shown in~\cite{GK94Bott}. Moreover, the Bott manifold \(\B_{\bm i}\) can be obtained as a toric degeneration of the Bott--Samelson variety \(Z_{\bm i}\) (see~\cite{Pasquier10}). This is the reason why we call such Bott manifolds of \emph{Bott--Samelson type}. 
\end{remark}

\begin{remark}\label{rmk_toric_Schubert}
	A family of Bott manifolds of BS type includes \emph{toric Schubert varieties} in a full flag variety of Lie type \(A\). For a simple algebraic group \(G\) and a Borel subgroup \(B\) of \(G\), the homogeneous space \(G/B\) is a smooth projective algebraic variety, called a \emph{full flag variety}. The Schubert varieties are subvarieties of \(G/B\) parameterized by the Weyl group \(W\) of \(G\). 
    The maximal torus \(T\) in \(B\) acts on \(G/B\) via the left multiplication, and the set of fixed points of this action can be parameterized by the Weyl group \(W\). For \(w \in W\), we denote by \(wB\) the corresponding fixed point in \(G/B\). The Schubert variety \(X_w\) is defined by the closure \(\overline{BwB/B}\) of the \(B\)-orbit of \(wB\) in \(G/B\). Any Schubert variety admits \(T\)-action induced by that on \(G/B\) and we say that \(X_w\) is \emph{toric} if \(X_w\) is a toric variety with respect to the action of a quotient of \(T\). 
    For a toric Schubert variety~\(X_w\) in \(G/B\), suppose that \(\bm{i}\) is a reduced decomposition of \(w\). Then the toric Schubert variety \(X_w\) is \(T\)-equivariantly isomorphic to the Bott--Samelson variety \(Z_{\bm i}\) (see~\cite[Chapter~18]{AndersonFulton24}). Moreover, in this case, \(X_w\) is isomorphic to \(\B_{\bm i}\).
\end{remark}

\section{Assemblies of ordered partitions}\label{section_AOP}

In this section, we introduce assemblies of ordered partitions. We
then find a surjection from the set of assemblies of ordered
partitions to the set of Bott matrices of BS type (see
\Cref{prop_sur}). This map will be used to enumerate the number
\(b(n,m)\) of Bott manifolds of BS type (see
\Cref{prop_equivalence_Bnm_POPnm} and \Cref{thm_Bott_matrices}).

\begin{definition}\label{def:2}
  An \emph{ordered partition} is a sequence
  \( \tau=(\tau^1,\dots,\tau^r) \) of mutually disjoint nonempty sets.
  We say that \( \{\tau^1,\dots,\tau^r\} \) is the \emph{underlying
    partition} of \( \tau \). We define
  \[
    U(\tau) = \tau^1 \cup \cdots \cup \tau^r, \qquad \min(\tau) = \min (U(\tau)).
  \]
\end{definition}

Note that the underlying partition of an ordered partition \( \tau \)
is a partition of \( U(\tau) \).

\begin{definition}\label{def_assembly_of_ordered_partitions}
  For a positive integer \( m \), an \emph{assembly of ordered
    partitions} of \( [m] \) is a sequence
  \( \sigma=(\sigma_1,\dots,\sigma_\ell) \) of ordered partitions satisfying the following conditions:
  \begin{enumerate}
  \item \( \{U(\sigma_1),\dots, U(\sigma_\ell)\} \) is a partition of \( [m] \),
  \item \( \min(\sigma_1) < \cdots < \min(\sigma_\ell) \).
  \end{enumerate}

  Let \( \sigma_a = (\sigma_a^1,\dots,\sigma_a^{r_a}) \)
  for \( a\in [\ell] \). We say that the assembly \( \sigma=(\sigma_1,\dots,\sigma_\ell) \) of ordered
  partitions has
  \emph{bound} \( n \) if 
\begin{equation}\label{eq_condition_on_r_and_n}
r_1+\dots+r_{\ell} + ({\ell}-1) \leq n.
\end{equation} 
We denote by \(\AOP(n,m)\) the set of assemblies of ordered partitions
of \([m]\) with bound \( n \).
\end{definition}

For example,
\(( (\{ 2,4,6 \}, \{ 1,5 \}), (\{ 3 \}, \{ 7,9 \}), (\{ 8\}) )\) is
an assembly of the ordered partition of~\([9]\). For
convenience, we use ``\(|\)'' to distinguish each block in an ordered
partition, that is,
\[
( (\{ 2,4,6 \}, \{ 1,5 \}), (\{ 3 \}, \{ 7,9 \}), (\{ 8\}) )
= ( (2 \, 4\, 6 \, | \, 1 \, 5), (3 \, | \, 7 \, 9 \, | \, 8)).
\]
Note that, for
\( \sigma = (\sigma_1,\dots,\sigma_\ell)\in \AOP(n,m) \), since
\( \min(\sigma_1) < \cdots < \min(\sigma_\ell) \), we can identify
\( \sigma \) with the set \( \{\sigma_1,\dots,\sigma_\ell\} \) of
ordered partitions.

Let \( \beta:[n]^m\to B(n,m) \) be the map defined by
\(\beta(\bm i) = B_{\bm i}\), which is surjective by definition. In what
follows, we define two maps \(\alpha \colon [n]^m \to \AOP(n,m)\) and
\(F \colon \AOP(n,m)\to B(n,m)\) such that the following diagram
commutes.
\begin{equation}\label{eq_commutes_psi_F_mu}
\begin{tikzcd}
	{[n]^m} \arrow[r, "\alpha"] \arrow[rd, "\beta"']
		& \AOP(n,m) \arrow[d, "F"] \\
	& B(n,m) 
\end{tikzcd}
\end{equation}

We define the map \(\alpha \colon [n]^m \to \AOP(n,m)\) as follows.

\begin{definition}\label{def:alpha}
  Consider a sequence \(\bm i = (i_1,\dots,i_m) \in [n]^m\). Let
  \( a_1,\dots,a_t \) be the integers appearing in \( \bm i \), where
  \( a_1 < \cdots <a_t \). For each \( j\in [t] \), let \( A_j \) be
  the set of indices of the entries equal to \( a_j \)'s in
  \( \bm i \), that is, \( A_j = \{k\in [m]: i_k = a_j\} \). We define
  the ordered partitions \( \sigma_1, \dots,\sigma_\ell \) as follows.
  First, \( \sigma_1 \) is the ordered partition
  \( (A_1,\dots,A_{r_1}) \), where \( r_1 \) is the largest integer
  such that \( a_1,\dots,a_{r_1} \) are consecutive integers. Next,
  \( \sigma_2 \) is the ordered partition
  \( (A_{r_1+1},\dots,A_{r_1+r_2}) \), where \( r_2 \) is the largest
  integer such that \( a_{r_1+1},\dots,a_{r_1+r_2} \) are consecutive
  integers. Continuing in this way, the ordered partitions
  \( \sigma_1,\dots,\sigma_\ell \) are defined until we have
  \( r_1 + \cdots + r_\ell = t \). Finally, we define
  \( \alpha(\bm i) \) to be
  \( \sigma = (\sigma_1,\dots,\sigma_\ell)\in \AOP(n,m) \).

\end{definition}

For example, for \(\bm i = (6,3,8,3,6,2,9) \in [9]^7\), we have
\( a_1=2 \), \( a_2=3 \), \( a_3=6 \), \( a_4=8 \), \( a_5=9 \), and
\( A_1 = \{6\} \), \( A_2 = \{2,4\} \), \( A_3=\{1,5\} \),
\( A_4= \{3\}\), \( A_5= \{7\}\); hence,
\[
\sigma(\bm i) = ((6 \, | \,2 \, 4 ), (1 \, 5), (3\, | \, 7)).
\]

The next lemma follows immediately from the definition of the map
\( \alpha \). 

\begin{lemma}\label{lemma_sigma_only_consider_differences}
  Let \(\bm i = (i_1,\dots,i_m) \in [n]^m\). For any positive integer
  \(k < \min\{i_1,\dots,i_m\}\), we have
  \(\sigma(\bm i) = \sigma(\bm i - k)\), where
  \(\bm i - k = (i_1 - k,i_2-k,\dots,i_m-k)\).
\end{lemma}

Note that, using the notation in the construction of
\( \alpha(\bm i) \) above, the largest integer appearing in
\( \bm i \) is at least \(r_1 + \cdots + r_{\ell} + ({\ell}-1)\).
Hence, for \(\bm i \in [n]^m\), the inequality
\eqref{eq_condition_on_r_and_n} holds, which implies that
\(\sigma(\bm i) \in \AOP(n,m)\).

\begin{definition}\label{def_neighbors}
  Let \(\sigma = (\sigma_1,\dots,\sigma_\ell) \in \AOP(n,m)\), where
  \( \sigma_i = (\sigma_i^1 \, |\, \cdots \,|\, \sigma_i^{r_i}) \).
  For integers \(1 \leq k < j \leq m\), we say that \(k\) and \(j\)
  are \emph{neighbors} in \(\sigma\) if they are in the same block of
  an ordered partition \(\sigma_a\), or if they appear in consecutive
  blocks of \(\sigma_a\). In other words, \(k\) and \(j\) are
  neighbors in \( \sigma \) if \(k \in \sigma_a^b \) and
  \( j \in \sigma_a^{b-1} \cup \sigma_a^b \cup \sigma_a^{b+1} \) for
  some \( a \) and \( b \).
\end{definition}

Now we define the map \( F \colon \AOP(n,m)\to B(n,m) \).

\begin{definition}\label{def:F}
Consider
\[
  \sigma = ((\sigma_1^1 \, |\, \cdots \,|\, \sigma_1^{r_1}),
  \dots,(\sigma_{\ell}^1 \,|\, \cdots \,|\, \sigma_{\ell}^{r_{\ell}}))
  \in \AOP(n,m).
\]
 For \(1 \leq k < j\leq m\), let
\[
B_{j,k} = 
\begin{cases}
	-2 & \text{ if }j,k \in \sigma_a^b \text{ for some \( a \) and \( b \)}, \\
	1 & \text{ if } j \in \sigma_a^b \text{ and }k \in \sigma_a^{b-1} \cup \sigma_a^{b+1}  \text{ for some \( a \) and \( b \)}, \\
	0 & \text{ otherwise}.
\end{cases}
\]
We also set \(B_{j,j} = -1\) for \(j \in [m]\) and \(B_{j,k} = 0\) for
\(1 \leq j < k \leq m\). Then we define \( F(\sigma) \) to be the
\( m\times m \) Bott matrix~\([B_{j,k}]\).
\end{definition}

Note that for \(1 \leq k < j \leq m\), we have \(B_{j,k} \neq 0\) if
and only if \(j\) and \(k\) are neighbors in \( \sigma \). The
following proposition shows that the diagram in
\eqref{eq_commutes_psi_F_mu} commutes, and consequently, the map
\( F \colon \AOP(n,m)\to B(n,m) \) is surjective.

\begin{proposition}\label{prop_sur}
  For any \( \bm i \in [n]^m \), we have
  \( F(\alpha(\bm i)) = \beta(\bm i) \). Moreover, the map
  \( F:\AOP(n,m) \to B(n,m) \) is surjective.
\end{proposition}
\begin{proof}
  Let \(\bm i = (i_1,\dots,i_m) \in [n]^m\), and let
  \[
    \alpha(\bm i) = \sigma = ((\sigma_1^1 \, |\, \cdots \,|\,
    \sigma_1^{r_1}), \dots,(\sigma_{\ell}^1 \,|\, \cdots \,|\,
    \sigma_{\ell}^{r_{\ell}})) .
  \]
  By the definition of the map \(\alpha\), for
  \(1 \leq k < j \leq m\), we have
\begin{enumerate}
	\item \(i_j = i_k\) if and only if \(j,k \in \sigma_a^b\) for some \( a \) and \( b \);
	\item \(|i_j - i_k| = 1\) if and only if \(j \in \sigma_a^b\)
and \(k \in \sigma_a^{b-1} \cup \sigma_a^{b+1}\) for some \( a \) and \( b \);
	\item \(|i_j - i_k| >1 \) if and only if \(j\) and \(k\) are
not neighbors in \(\sigma\).
\end{enumerate}
Thus, the definition of the matrix \( F(\sigma) = [B_{j,k}] \) in
\Cref{def:F} is identical to that of
\( \beta(\bm i) = B_{\bm i} \) in \Cref{def_Bott_manifold_of_BS_type}.
Accordingly, we obtain \(F(\alpha(\bm i)) = \beta(\bm i)\).

Since \( \beta = F \circ \alpha\) and \(\beta : [n]^m \to B(n,m)\) is
surjective, the map \( F:\AOP(n,m) \to B(n,m) \) is also surjective.
\end{proof}

\section{Enumeration of Bott manifolds of Bott--Samelson type}\label{section_enumeration}

In this section, we prove \Cref{thm_Bott_matrices}, which relates Bott
manifolds of BS type and assemblies of ordered partitions. As an
application, we give a proof of Corollary~\ref{cor_enumeration},
which enumerates the number \(b(n,m)\) of Bott manifolds of BS type of
dimension \(m\) determined by \(\bm i \in [n]^m\).

We begin with some definitions. For an ordered partition
\(\eta = (\eta^1,\dots,\eta^r)\), we define
\(\overline{\eta} = (\eta^r,\dots,\eta^1)\). We define an equivalence
relation \(\sim\) on \(\AOP(n,m)\) as follows.

\begin{definition}
  Let \(\sigma = (\sigma_1,\dots,\sigma_{\ell})\) and
  \(\tau = (\tau_1,\dots,\tau_{\ell'})\) be elements in \(\AOP(n,m)\).
  We define \(\sigma \sim \tau\) if \(\ell = \ell'\) and
  \(\sigma_a = \tau_a\) or \(\sigma_a = \bar{\tau}_a\) for all
  \(a = 1,\dots,\ell\).
\end{definition}

\begin{example}\label{example_relation_on_POP}
  There are \(14\) equivalence classes in \(\AOP(5,3)/\sim\) as
  follows:
\[
\begin{array}{lll}
\{(1,2,3)\}, \\
\{(1 \, 2 \, | \, 3)\} \sim \{(3 \, | \, 1 \, 2)\}, & 
\{(1 \, 3 \, | \, 2)\} \sim \{(2 \, | \, 1 \, 3)\}, &
\{(2 \, 3 \, | \, 1)\} \sim \{(1 \, | \, 2 \, 3)\}, \\
\{(1 \, | \, 2 \, | \, 3)\} \sim \{(3 \, | \, 2 \, | \, 1)\}, & 
\{(1 \, | \, 3 \, | \, 2)\} \sim \{(2 \, | \, 3 \, | \, 1)\}, & 
\{(2 \, | \, 1 \, | \, 3)\} \sim \{(3 \, | \, 1 \, | \, 2)\}, \\
\{ (1 \, 2), (3)\}, & \{(1 \, 3), (2)\}, & \{(2 \, 3), (1)\}, \\
\{(1 \, | \, 2), (3)\} \sim \{(2 \, | \, 1), (3)\}, & 
\{(1 \,|\,3),(2)\} \sim \{(3\,|\,1),(2)\}, & 
\{(2\,|\,3),(1)\} \sim \{(3\,|\,2),(1)\}, \\
\{(1),(2),(3)\}.
\end{array}
\]
\end{example}

The following proposition gives a connection between Bott matrices of
BS type and assemblies of ordered partitions.

\begin{proposition} \label{prop_equivalence_Bnm_POPnm}
	There is a bijection between the set \(B(n,m)\) of Bott matrices of BS type and the set \(\AOP(n,m)/\sim\) of equivalence classes of assemblies of ordered partitions. 
\end{proposition}

\begin{proof}
  Recall from Proposition~\ref{prop_sur} that the map
  \(F \colon \AOP(n,m) \to B(n,m)\) is surjective. Hence, it suffices
  to prove that for \(\sigma ,\tau \in \AOP(n,m)\), we have
  \(F(\sigma) = F(\tau)\) if and only if \(\sigma \sim \tau\).

  Suppose that \(\sigma \sim \tau\). By the definition of \(\sim\),
  the neighborhood relations in \(\sigma\) and \(\tau\) are the same:
  for \(1 \leq k < j \leq m\), we have that \(j\) and \(k\) are
  neighbors in \(\sigma\) if and only if they are neighbors in
  \(\tau\). Moreover, \(k\) and \(j\) are in the same block in
  \(\sigma\) if and only if they are also in the same block in
  \(\tau\). Hence, by the definition of the map \( F \) in
  \Cref{def:F}, we have \(F(\sigma) = F(\tau)\).

  Suppose now that \( F(\sigma) = F(\tau) = [B_{j,k}]\). Considering
  the pairs \( (j,k) \) with \( B_{j,k} = -2 \), we obtain
  \[
    \{\sigma_1^1,\dots,\sigma_1^{r_1},\dots,\sigma_{\ell}^{1},\dots,\sigma_{\ell}^{r_{\ell}}\}
    = \{\tau_1^1,\dots,\tau_1^{r'_1},\dots,\tau_{\ell'}^1,\dots,\tau_{\ell'}^{r'_{\ell'}}\}.
  \]
  We call a sequence \( (c_1,\dots,c_r) \)
  of integers a \emph{path} if the following condition holds:
  \begin{enumerate}
  \item \( B_{c_i,c_j}\ne -2 \) for all \( i\ne j \),
  \item \( B_{c_i,c_{i+1}} = 1 \) for all \( i\in [r-1] \).
  \end{enumerate}
  A path is \emph{maximal} if it cannot be extended to a longer path.

  Let \( (c_1,\dots,c_r) \) be a maximal path containing \( 1 \).
  Since \( \min(\sigma_1) < \cdots < \min(\sigma_\ell) \), we have
  \( 1\in \sigma_1^1 \cup \cdots \cup \sigma_1^{r_1} \). Hence, by the
  construction of \( F(\sigma) = [B_{j,k}] \), we obtain \( r=r_1 \),
  and \( c_i\in \sigma_1^i \) for all \( i\in [r_1] \) or
  \( c_i\in \sigma_1^{r_1+1-i} \) for all \( i\in [r_1] \). Similarly,
  considering \( F(\tau) = [B_{j,k}] \), we obtain \( r=r'_1 \), and
  \( c_i\in \tau_1^i \) for all \( i\in [r'_1] \) or
  \( c_i\in \tau_1^{r'_1+1-i} \) for all \( i\in [r'_1] \). This shows
  that \( r_1=r'_1 \), and \( \sigma_1 = \tau_1 \) or
  \( \sigma_1 = \overline{\tau}_1 \). Now, let \( (d_1,\dots,d_s) \)
  be a maximal path containing
  \( \min ([m]\setminus (\sigma_1^1 \cup \cdots \cup \sigma_1^{r_1}))
  \). By the same argument, we obtain that \( r_2=r'_2 \), and
  \( \sigma_2 = \tau_2 \) or \( \sigma_2 = \overline{\tau}_2 \).
  Continuing in this way, we obtain \( \sigma \sim \tau \), which
  completes the proof.
\end{proof}

Since there is a bijection between
\( \{\B_{\bm i}: \bm i \in [n]^m\} \) and \( B(n,m) \), we obtain
\Cref{thm_Bott_matrices} immediately from
\Cref{prop_equivalence_Bnm_POPnm}. We now prove \Cref{cor_enumeration}
using the argument of the exponential formula
(see~\cite[Corollary~5.1.6]{Stanley_EC_2}).

\begin{proof}[Proof of \Cref{cor_enumeration}]
  By \Cref{thm_Bott_matrices}, we have
  \( b(n,m) = |\AOP(n,m)/\sim| \). Recall that \( \AOP(n,m) \) is the
  set of assemblies \( \sigma=(\sigma_1,\dots,\sigma_\ell) \) of
  ordered partitions
  \( \sigma_i = (\sigma_i^1,\dots,\sigma_i^{r_i}) \) such that
  \( \{\sigma_i^j: i\in [\ell], j\in [r_i]\} \) is a partition of \( [m] \) and
  \[
    (r_1+1)+\dots+(r_{\ell}+1) \leq n+1.
  \]
  Let \( \AOP'(n,m) \) be the set of such assemblies of ordered
  partitions satisfying
  \begin{equation}\label{eq:8}
    (r_1+1)+\dots+(r_{\ell}+1) = n+1,
  \end{equation}
  and let \( b'(n,m) = |\AOP'(n,m)/\sim| \). Observe that
  \( \AOP(n,m) = \bigcup_{i=1}^n \AOP'(i,m) \) holds, and for any
  \( \sigma,\tau\in \AOP(n,m) \), if \( \sigma\in \AOP'(i,m) \) and
  \( \sigma\sim \tau \), then \( \tau\in \AOP'(i,m) \). Thus, for
  any \( m\ge1 \), we have
  \[
    \sum_{n\ge1} b(n,m) y^n
    = \sum_{n\ge1} \left( \sum_{i=1}^n b'(i,m)  \right) y^n
    = \sum_{i\ge1} b'(i,m) y^i \sum_{n\ge i}  y^{n-i}
    = \frac{1}{1-y} \sum_{i\ge1} b'(i,m) y^i .
  \]
  Hence,  it suffices to prove that
  \begin{equation}\label{eq:10}
  \sum_{m\ge1} \left( \sum_{n\geq 1} b'(n,m) y^{n+1} \right) \frac{x^m}{m!} 
= \exp \left( \frac{y}{2}
\left( \frac{1}{1-y(e^x-1)} +y(e^x-1) -1 \right)\right) - 1.
\end{equation}

  For two ordered partitions \( \pi \) and \( \rho \), we define
  \( \pi\sim \rho \) if \( \pi = \rho \) or
  \( \pi = \overline{\rho} \). Let \( f(m,k) \) be the number of
  ordered partitions of \([m]\) with \(k\) blocks. Then, for any
  positive integer \( k \),
\begin{equation}\label{eq:6}
  \sum_{m \geq 1} f(m,k)\frac{x^m}{m!}
  = \sum_{m \geq 1} \left( \sum_{\substack{i_1,\dots,i_k\ge1\\ i_1 + \cdots + i_k=m}}
    \frac{m!}{i_1! \cdots i_k!}
  \right)\frac{x^m}{m!}
  =  (e^x -1)^k. 
\end{equation}
Let \(\tilde{f}(m,k)\) be the number of equivalence classes of ordered
partitions of \([m]\) with \(k\) blocks under the relation \(\sim\).
By definition, we have \( \tilde{f}(m,k) = f(m,k)/2 \) if \( k\ge2 \),
and \( \tilde{f}(m,k) = f(m,k) \) if \( k=1 \). Hence, by
\eqref{eq:6}, we have
\begin{align}
  \notag
\sum_{m,k \geq 1} \tilde{f}(m,k)\frac{x^m}{m!} y^{k+1} 
&= (e^x-1)y^2 + \sum_{k\geq 2} \frac{1}{2} (e^x-1)^k y^{k+1}  \\
  \notag
& = \frac{y}{2} \left(
\frac{1}{1-y(e^x-1)} - y(e^x-1)-1
\right) + y^2(e^x-1) \\
  \label{eq:9}
&= \frac{y}{2}
\left( \frac{1}{1-y(e^x-1)} +y(e^x-1) -1 \right). 
\end{align}

Now consider \( \sigma=(\sigma_1,\dots,\sigma_\ell)\in \AOP'(n,m) \).
Observe that the equivalence class \( [\sigma]\in \AOP'(n,m)/\sim \)
can be identified with the set
\( \{[\sigma_1],\dots,[\sigma_\ell]\} \) of the equivalence classes
\( [\sigma_i] \) of ordered partitions under the relation \( \sim \).
Such a set is constructed as follows.
\begin{enumerate}
\item Partition \( [m] \) into nonempty blocks \( A_1,\dots,A_\ell \)
  for some \( \ell\ge1 \), where
  \( \min(A_1) < \cdots < \min(A_\ell) \). The number of ways to do
  this is
  \[
   \sum_{\ell\ge1} \sum_{a_1 + \cdots + a_\ell = m} \frac{1}{\ell!} \binom{m}{a_1,\dots,a_\ell}.
  \]
\item For each \( i\in [\ell] \), create an equivalence class
  \( [\sigma_i] \) of ordered partitions, where \( \sigma_i \) is an
  ordered partition of \( A_i \) with \( r_i \) nonempty blocks, such
  that \eqref{eq:8} holds. The number of ways to do this is
  \[
    \sum_{\substack{r_1,\dots,r_\ell\ge1\\ (r_1+1) + \cdots + (r_\ell+1) = n+1}} \tilde{f}(a_1,r_1) \cdots \tilde{f}(a_\ell,r_\ell).
  \]
\end{enumerate}
This shows that for any \( m\ge1 \),
\[
  \sum_{n\geq 1} b'(n,m) y^{n+1}
  = \sum_{\ell\ge1} \sum_{a_1 + \cdots + a_\ell = m} \frac{1}{\ell!} \binom{m}{a_1,\dots,a_\ell}
 \prod_{i=1}^{\ell} \left( \sum_{r_i \ge1}  \tilde{f}(a_i,r_i) y^{r_i+1} \right).
\]
Therefore, by \eqref{eq:9}, we have
\begin{align*}
  \sum_{m\ge1} \left( \sum_{n\geq 1} b'(n,m) y^{n+1} \right) \frac{x^m}{m!} 
  &= \sum_{\ell\ge1}\frac{1}{\ell!}  \prod_{i=1}^\ell
 \left( \sum_{a_i,r_i\ge1} \tilde{f}(a_i,r_i) \frac{x^{a_i}}{a_i!} y^{r_i+1} \right)  \\
&= \exp \left( \frac{y}{2}
\left( \frac{1}{1-y(e^x-1)} +y(e^x-1) -1 \right)\right) - 1.
\end{align*}
Hence, we obtain \eqref{eq:10}, which completes the proof.
\end{proof}

In \Cref{tablebnm}, we display the first few terms of \(b(n,m)\). 
\begin{table}
\begin{tabular}{c|rrrrrrr}
\toprule 
\backslashbox{\(m\)}{\(n\)} & 1 & 2 & 3 & 4 & 5 & 6 & 7\\
\midrule 
1 & 1 & 1 & 1 & 1 & 1 & 1 & 1  \\
2 & 1 & 2 & 3 & 3 & 3 & 3 & 3\\
3 & 1 & 4 & 10 & 13 & 14 & 14 & 14\\
4 & 1 & 8 & 33 & 63 & 84 & 90 & 91 \\
\bottomrule 
\end{tabular}
\caption{The first few terms of the number \(b(n,m)\) of Bott matrices in \(B(n,m)\).}
\label{tablebnm}
\end{table}

\section{Isomorphism classes of Bott manifolds of Bott--Samelson type}\label{section_isomorphism}

In this section, we reinterpret the two operations on Bott matrices
introduced in~\cite{CLMP_unique}, called \op~and \oop, in terms of
assemblies of ordered partitions.
Using this interpretation, we prove Theorem~\ref{thm_isom}.

Let \( m>0 \) be a fixed integer, and let \(B\) be a Bott matrix of
size \(m \times m\), with column vectors
\(\mathbf{v}_1,\dots,\mathbf{v}_m\). For \(I \subseteq [m]\), consider
the \(m \times m\) lower triangular matrix \(L_I\) whose \(j\)th
column vector \(\mathbf{c}_j\) is defined by
\[
\mathbf{c}_j \colonequals \begin{cases}
\mathbf{e}_j & \text{ if }j \in I, \\
\mathbf{v}_j & \text{ if } j \notin I. 
\end{cases}
\]
By~\cite[Proposition~3.1]{CLMP_unique}, the matrix
\(B_I \colonequals L_I^{-1} L_{I^c}\) is again a Bott matrix. We call
the map that sends \( B \) to \( B_I \) for some \( I\subseteq [m] \)
the \emph{first operation} \(\op\).

For \( i\in [m-1] \), we denote by \(\dot{s}_i\) the permutation matrix of
the simple transposition \( s_i = (i,i+1) \in \mathfrak{S}_m\), that
is, the matrix whose \((s_i(j), j)\)-entry is \( 1 \) for
\(j=1,\dots,m\), and all other entries are \( 0 \). If \( \dot{s}_i B \dot{s}_i \)
is again a Bott matrix, then we call the map that sends \( B \) to
\( \dot{s}_i B \dot{s}_i \) for some \(i \in [m-1]\) the \emph{second operation}
\(\oop\).
One can study the isomorphism classes of Bott manifolds using these two operations: 
\begin{proposition}[{\cite[Proposition~3.4]{CLMP_unique}}]\label{prop_operations_and_isom}
  Let \(\B\) and \(\B'\) be Bott manifolds with the corresponding Bott
  matrices \(B\) and \(B'\), respectively. Then \(\B\) and \(\B'\) are
  isomorphic \textup{(}as toric varieties\textup{)} if and only if
  \(B\) can be obtained from \(B'\) by applying a sequence of the two
  operations \(\op\)~and \(\oop\).
\end{proposition}

We recall a brief idea of the proof of
Proposition~\ref{prop_operations_and_isom} for the reader's
convenience. Note that the fan of a Bott manifold of dimension \(m\)
is combinatorially equivalent to the normal fan of a cube \([0,1]^m\),
and each pair \(\{\mathbf{e}_j, \mathbf{v}_j\}\) of ray generators
corresponds to two rays that do not form a cone (cf.
Section~\ref{section_Bott_manifolds_of_BS_type}). Since a Bott
manifold is smooth, the set of ray generators forming a maximal cone
constitutes a \(\Z\)-basis of the lattice. The first operation~\(\op\)
corresponds to the choice of a \(\Z\)-basis, and the second
operation~\(\oop\) corresponds to the choice of ordering of that
basis.

We will consider operations on the set of Bott matrices of BS type to study the isomorphism classes. 

\begin{example}\label{example_op1}
  Let \( m=3 \) and \(\bm i = (3,1,2)\). Then
  the matrix \(B = B_{\bm i}\) is given by 
	\[
	B =  \begin{bmatrix}
		-1 & 0 & 0 \\
		0 & -1 & 0 \\
		1 & 1 & -1 
	\end{bmatrix}.
	\]
	If \(I =\emptyset\), we have
        \(L_{I} = L_{\emptyset} = B \) and
        \(L_{I^c} = L_{\{1,2,3\}} = [\mathbf{e}_1 \ \mathbf{e}_2 \
        \mathbf{e}_3]\). Accordingly, we obtain
	\[
	B_{\emptyset} = B^{-1} = \begin{bmatrix}
		-1 & 0 & 0 \\
		0 & -1 & 0 \\
		-1 & -1 & -1
	\end{bmatrix}.
	\]
	Considering all subsets \( I \) of \([3]\), we obtain the
        following table.
	\begin{center}
		\begin{tabular}{c|cccccccc}
			\toprule 
			\(I\) & \(\emptyset\) & \(\{1\}\) & \(\{2\}\) & \(\{3\}\) & \(\{1,2\}\) & \(\{1,3\}\) & \(\{2,3\}\) & \(\{1,2,3\}\)\\
			\midrule 
			\(L_I^{-1} L_{I^c}\) & \(B^{-1}\) & \(B^{-1}\) & \(B^{-1}\)  
			& \(B\) & \(B^{-1}\) & \(B\) & \(B\) & \(B\) \\
			\bottomrule 
		\end{tabular}
	\end{center}
\end{example}

Observe that in Example~\ref{example_op1}, the subsets
\(\emptyset, \{1\}, \{2\}\), and \( \{1,2\} \) produce the Bott matrix
\( B^{-1} \), which is \emph{not} contained in \(B(n,m)\) since the
\((3,1)\)-entry is \(-1\). On the other hand, the other subsets
provide the same matrix as \(B\). Hence, in this example, the map
\( B\mapsto B_I \) produces a matrix in \(B(n,m)\) if and only if
\(3\in I\). Moreover, \(\{3\}\) is the set of row indices with nonzero
entries below the main diagonal of \( B \). The following lemma shows
that this phenomenon holds in general.

\begin{lemma}\label{lemma_op1_not_in_B}
  Let \(B = [b_{j,k}] \in B(n,m)\), and let
  \[
    J_B \colonequals \{j \in [m] \mid \text{there exists \(k < j\) such that }b_{j,k} \neq 0\}.
  \]
  \begin{enumerate}
  \item If \(J_B \not\subseteq I\), we have \( B_I \notin B(n,m) \).
  \item If \(J_B \subseteq I\), we have \( B_I = B \in B(n,m) \). 
  \end{enumerate}
\end{lemma}
\begin{proof}
  For the first statement, suppose that \(J_B \not\subseteq I\). Let
  \(p = \min (J_B \setminus I)\) and
  \(q = \min \{ k \in [m] : b_{p,k} \neq 0\}\). Note that since
  \( p\in J_B \), we have \( q<p \). We claim that the \((p,q)\)-entry
  of \(B_I\) is \(-b_{p,q}\). This implies \( B_I \notin B(n,m) \)
  because its entry \(-b_{p,q}\), which is either \(2\) or \(-1\), is
  below the diagonal and is not an element of \( \{0,1,-2\} \). Hence,
  it suffices to prove the claim.
  
  Let \(\mathbf{v}_1,\dots,\mathbf{v}_m\) be the column vectors of
  \( B \), and let \(\mathbf{r}_1,\dots,\mathbf{r}_m\) be the row
  vectors of \(L_I^{-1}\). Then the \((p,q)\)-entry of
  \(B_I = L_I^{-1} L_{I^c}\) is
  \begin{equation}\label{eq:2}
    (B_I)_{p,q} = \begin{cases}
      \mathbf{r}_{p} \cdot \mathbf{v}_q & \text{ if }q \in I, \\
      \mathbf{r}_{p} \cdot \mathbf{e}_q & \text{ if } q \notin I. 
    \end{cases}
  \end{equation}
  Since \(L_I^{-1}L_I \) is the identity matrix, for every \( t\in [m] \), we have
  \begin{equation}\label{eq:3}
    \delta_{p,t} = (L_I^{-1}L_I)_{p,t} = 
    \begin{cases}
      \mathbf{r}_{p} \cdot \mathbf{e}_t  & \text{ if }t \in I,\\
      \mathbf{r}_{p} \cdot \mathbf{v}_t  & \text{ if }t \notin I.
    \end{cases}
  \end{equation}
  By adding \eqref{eq:2} and \eqref{eq:3} with \( t=q \), we obtain
  \begin{equation}\label{eq:1}
    (B_I)_{p,q} =  \mathbf{r}_p \cdot (\mathbf{e}_q+ \mathbf{v}_q).
  \end{equation}
  Since the column vectors of \(L_I\) form a basis of \(\R^m\), there
  exist \(c_i,d_i\in \R\) such that
  \begin{equation}\label{eq:4}
    \mathbf{e}_q + \mathbf{v}_q = 
    \sum_{i \notin I} c_i \mathbf{v}_i 
    + \sum_{i \in I} d_i \mathbf{e}_i.
  \end{equation}

  By \eqref{eq:3} and the fact that \( p\notin I \), taking the dot
  product with \( \mathbf{r}_p \) on both sides of \eqref{eq:4} gives
  \begin{equation}\label{eq:cp}
      \mathbf{r}_p \cdot (\mathbf{e}_q+ \mathbf{v}_q) = c_p.
  \end{equation}
  On the other hand, since \( q<p \) and \( p\not\in I \),
  taking the dot product with \( \mathbf{e}_p \) on
  both sides of \eqref{eq:4} gives
  \begin{equation}\label{eq:5}
    \mathbf{e}_p \cdot \mathbf{v}_q
    = \sum_{i \in I^c} c_i \mathbf{e}_p \cdot \mathbf{v}_i 
    = \sum_{i \in I^c\cap J_B} c_i \mathbf{e}_p \cdot \mathbf{v}_i
    + \sum_{i \in I^c \setminus J_B} c_i \mathbf{e}_p \cdot \mathbf{v}_i.
  \end{equation}
  Note that \( \mathbf{e}_j \cdot \mathbf{v}_k = b_{j,k} \). If
  \( i \in I^c\cap J_B = J_B \setminus I \), then
  \( i\ge p = \min (J_B \setminus I) \). Since \( B \) is lower
  triangular,
  \( \mathbf{e}_p\cdot \mathbf{v}_i = b_{p,i} = -\delta_{i,p} \). If
  \( i\in I^c \setminus J_B \), then since \( i\notin J_B \), we have
  \( b_{j,i} = 0 \) for all \( j\ne i \). Since \( i\ne p \), we have
  \( \mathbf{e}_p\cdot \mathbf{v}_i = b_{p,i} = 0 \). Thus, we obtain
  from \eqref{eq:5} that
  \begin{equation}\label{eq:-cp}
    b_{p,q} = \mathbf{e}_p\cdot \mathbf{v}_q = - c_p.
  \end{equation}
  Combining \eqref{eq:1}, \eqref{eq:cp}, and \eqref{eq:-cp}, we obtain
  \( (B_{I})_{p,q} = -b_{p,q} \), as claimed.

  For the second statement, suppose \(J_B \subseteq I\). We claim that
  \( L_I^2 = E_m\). Let \(\mathbf{u}_1,\dots,\mathbf{u}_m\) be the
  column vectors of \(L_I^2\). By the definition of \(L_I\), we have
  \( \mathbf{u}_k = \mathbf{e}_k \) for all \( k\in I \). Suppose that
  \(k \notin I\). Note that for two \( m\times m \) matrices \( A \)
  and \( B \), the \( k \)th column of \( AB \) is the linear
  combination of the columns of \( A \) with coefficients given by the
  \(k\)th column vector of \( B \). Therefore, we have
  \[
    \mathbf{u}_k =
    \sum_{j \in I}(\mathbf{e}_j \cdot \mathbf{v}_k) \mathbf{e}_j + \sum_{j \notin I} (\mathbf{e}_j \cdot \mathbf{v}_k) \mathbf{v}_j.
  \]
  Since \(J_B \subseteq I\), we have
  \( \mathbf{e}_j\cdot \mathbf{v}_k = b_{j,k} = - \delta_{j,k}\) for
  all \(j \notin I \). Hence, by the above equation, we obtain
  \[
    \mathbf{u}_k =
    \sum_{j \neq k} (\mathbf{e}_j \cdot \mathbf{v}_k) \mathbf{e}_j - \mathbf{v}_k 
    = (\mathbf{v}_k + \mathbf{e}_k) - \mathbf{v}_k = \mathbf{e}_k.
  \]
  Thus, \( \mathbf{u}_k = \mathbf{e}_k \) for all \( k\in [m] \), and
  we obtain \( L_I^2 = E_m \), as claimed. Equivalently,
  \( L_I^{-1} = L_I \).

  Now let \( \mathbf{w}_1,\dots,\mathbf{w}_m \) be the columns of
  \(B_I = L_I^{-1} L_{I^c} \). By the claim, we have
  \( L_I L_{I^c} = [ \mathbf{w}_1 \ \cdots \ \mathbf{w}_m ] \). If
  \( k\notin I \), the \( k \)th columns of \( L_I \) and
  \( L_{I^c} \) are \( \mathbf{v}_k \) and \( \mathbf{e}_k \),
  respectively; hence \( \mathbf{w}_k = \mathbf{v}_k \). Suppose that
  \(k \in I\). Then, by the same arguments above, we obtain
  \[
   \mathbf{w}_k = \sum_{j \in I}(\mathbf{e}_j \cdot \mathbf{v}_k) \mathbf{e}_j + \sum_{j \notin I}(\mathbf{e}_j \cdot \mathbf{v}_k)\mathbf{v}_j = \mathbf{v}_k.
  \]
  Therefore, \( \mathbf{w}_k = \mathbf{v}_k \) for all \( k\in [m] \),
  and we obtain the second statement \( B_I = B \).
\end{proof}

By Lemma~\ref{lemma_op1_not_in_B}, we obtain the following proposition. 
\begin{proposition}\label{prop_op1_on_BS}
	Let \(B \in B(n,m)\) be a Bott matrix of BS type. For a subset \(I \subseteq [m]\), after applying the first operation \op, we have \(B_I = B\) or \(B_I \notin B(n,m)\). 
\end{proposition}

Proposition~\ref{prop_op1_on_BS} tells us that the first operation
acts~{\op} on the set of Bott matrices of BS type trivially. To
understand the second operation~{\oop}, we introduce some definitions.

\begin{definition}\label{def_admissible}
  Let \(\sigma \in \AOP(n,m)\). For a simple transposition
  \(s_j = (j,j+1) \in \mathfrak{S}_m\), we denote by
  \(s_j \cdot \sigma\) the assembly of an ordered partition obtained
  from \(\sigma\) by swapping the two integers \(j\) and \(j+1\). We
  say that \(s_j\) is an \emph{admissible} transposition of \(\sigma\)
  if \(j\) and \(j+1\) are not neighbors in \(\sigma\).
\end{definition}

\begin{definition}
  For \( \sigma,\tau\in \AOP(n,m)\), we define \(\sigma \approx \tau\)
  if there exist
  \( \rho^{(0)},\rho^{(1)},\dots,\rho^{(k)}\in \AOP(n,m) \) such that
  \( \rho^{(0)}\sim \sigma \), \( \rho^{(k)} \sim \tau \), and for
  each \( i\in [k] \), there is an admissible transposition
  \( s_{j_i} \) of \( \rho^{(i-1)} \) such that
  \( s_{j_i}\cdot \rho^{(i-1)} = \rho^{(i)} \).
\end{definition}

For example, we have
\( ( (1 \, 2), (3)) \approx ((1 \, 3), (2)) \approx ((2 \, 3), (1))
\).
Recall the map \( F \colon \AOP(n,m)\to B(n,m) \) given in
\Cref{def:F}. The following proposition describes the case where 
the second operation \oop~ can be applied to \(F(\sigma)\).

\begin{proposition}\label{prop_op2_on_BS}
  Let \(\sigma \in \AOP(n,m)\) and \( i\in [m-1] \). Then
  \( \dot{s}_i F(\sigma) \dot{s}_i \in B(n,m)\) if and only if \(s_i\) is an
  admissible transposition of~\(\sigma\).
\end{proposition}

\begin{proof} 
  Let \(F(\sigma) = [b_{j,k}]\). Since
  \( \dot{s}_i F(\sigma) \dot{s}_i \in B(n,m)\) is the matrix obtained from
  \( [b_{j,k}] \) by swapping the \( i \)th row and the \( (i+1) \)st
  row and swapping the \( i \)th column and the \( (i+1) \)st column,
  we have \( \dot{s}_i F(\sigma) \dot{s}_i \in B(n,m)\) if and only if
  \(b_{i+1,i} = 0\). By the definition of the map \( F \), we have
  \(b_{i+1,i} = 0\) if and only if \(i\) and \(i+1\) are not
  neighbors. This proves the proposition.
\end{proof}

We are now ready to prove Theorem~\ref{thm_isom}, which states that
there is a bijection from \( \AOP(n,m)/\approx \) to
\( \{\B_{\bm i} : \bm i \in [n]^m \}/(\text{isom}) \).

\begin{proof}[Proof of Theorem~\ref{thm_isom}]
  For a Bott matrix \( B\in B(n,m) \), let \( \B(B) \) denote the Bott
  manifold corresponding to \( B \). Define
  \( \phi: (\AOP(n,m)/\approx) \to \{\B_{\bm i} : \bm i \in [n]^m
  \}/(\text{isom}) \) by \( \phi([\sigma]) = [\B(F(\sigma))] \), where
  \( [\sigma] = \{\tau\in \AOP(n,m): \tau \approx \sigma\}\), and
  \( [\B(F(\sigma))] \) is the isomorphism class containing
  \( \B(F(\sigma)) \). By \Cref{prop_op1_on_BS,prop_op2_on_BS}, the
  map \( \phi \) is well defined. By Proposition~\ref{prop_sur}, the
  map \(F \colon \AOP(n,m)\to B(n,m)\) is surjective, and hence,
  \( \phi \) is also surjective.

  Suppose that \( \phi([\sigma]) = \phi([\tau]) \). Then the Bott
  manifolds \( \B(F(\sigma)) \) and \( \B(F(\tau)) \) are isomorphic.
  Thus, by Proposition~\ref{prop_operations_and_isom}, the Bott matrix
  \( F(\sigma) \) is obtained from \( F(\tau) \) by applying a
  sequence of the operations \op~ and \oop. By
  \Cref{prop_op1_on_BS,prop_op2_on_BS}, this implies
  \( \sigma \approx \tau \), that is \( [\sigma]=[\tau] \). Therefore,
  \( \phi \) is injective, which completes the proof.
\end{proof}

Now we introduce the notion of the indecomposability of Bott
manifolds. We note that in some literature, indecomposability has
different meanings (cf.~\cite{ch-ma-ou17}).

\begin{definition}
  We say that a Bott manifold is \emph{indecomposable} if it is not
  isomorphic to a product of lower-dimensional Bott manifolds.
  Otherwise, we say that it is \emph{decomposable}.
  Similarly, we say that a Bott matrix is \emph{indecomposable} or
  \emph{decomposable} if the corresponding Bott manifold is
  indecomposable or decomposable, respectively.
\end{definition}

Recall from~\cite[Proposition~3.1.14]{CLS11Toric} that one can describe the fan of a product of two toric varieties using those of toric varieties. Let \(\Sigma_1  \subseteq (N_1)_{\R}\) and \(\Sigma_2 \subseteq (N_2)_{\R}\) be fans, where \(N_1\) and \(N_2\) are lattices. Then 
  \[
    \Sigma_1 \times \Sigma_2 = \{\sigma_1 \times \sigma_2 : \sigma_i \in \Sigma_i\}
  \] 
  is a fan in \((N_1)_{\R} \times (N_2)_{\R} = (N_1 \times N_2)_{\R}\) and 
  \[
    X_{\Sigma_1 \times \Sigma_2} \cong X_{\Sigma_1} \times X_{\Sigma_2}. 
  \]
Considering Bott manifolds, we can see that a Bott manifold
\(\B(B)\) is decomposable if and only if after applying a
certain sequence of two operations \op~and \oop, the Bott
matrix \(B\) can be expressed by a block diagonal matrix whose
diagonal blocks are lower-dimensional Bott matrices.
Hence, we have the following lemma.

\begin{lemma}\label{lemma_decomposable_Bott_matrix}
  A Bott matrix \(B\) of size \( m\times m \) is decomposable if and
  only if there is a sequence of operations \op~and \oop~whose
  application to \( B \) yields a block diagonal matrix whose diagonal
  blocks are Bott matrices of size less than \( m\times m \).
\end{lemma}

Let \( \sigma=(\sigma_1,\dots,\sigma_\ell)\in \AOP(n,m) \). Then each
\( \sigma_a \) is an ordered partition whose underlying partition is a
partition of \( \{b+1,\dots,b+r\} \) for some \( b,r\ge1 \). Let
\( \sigma_a - b \) denote the ordered partition obtained from
\( \sigma_a \) by replacing each element \( i \) by \( i-b \). For
simplicity, we write \( (\sigma_a) \) to mean the assembly
\( (\sigma_a-b) \) of ordered partitions consisting of a single
ordered partition \( \sigma_a - b \). We also write \( F(\sigma_a) \)
in place of \( F((\sigma_a)) \).

\begin{proposition}\label{prop_indecomposable}
  Let \(\bm i \in [n]^m\) and
  \(\alpha(\bm i) = (\sigma_1,\dots,\sigma_{\ell})\). Then we have
  \[
    \B_{\bm i} \cong \B(F(\sigma_1)) \times \dots \times \B(F(\sigma_{\ell})).
  \]
  Moreover, each \(\B(F(\sigma_{a}))\) is indecomposable.
\end{proposition}
\begin{proof}
  Recall that for \(1 \leq k < j \leq m\), the \((j,k)\)-entry of
  \(B_{\bm i}\) is zero if and only if \(j\) and \(k\) are not
  neighbors in \( \alpha(\bm i) \). Accordingly, if
  \(j \in \sigma_{a_1}\) and \(k \in \sigma_{a_2}\) with
  \(a_1 \neq a_2\), then the \((j,k)\)-entry of \(B_{\bm i}\) is
  \(0\). Therefore, there is a sequence of operations \op~and
  \oop~whose application to \(F(\alpha(\bm i))\) yields a block
  diagonal matrix whose diagonal blocks are
  \(F(\sigma_1),\dots,F(\sigma_{\ell})\). Then by
  \Cref{lemma_decomposable_Bott_matrix}, the Bott manifold
  \(\B_{\bm i}\) is the product of
  \(\B(F(\sigma_1)), \dots,\B(F(\sigma_{\ell}))\).
	
  Now we prove that each \(\B(F(\sigma_{a}))\) is indecomposable.
  Suppose that the underlying partition of \( \sigma_a \) is a
  partition of \( \{b+1,\dots,b+r\} \). Then
  \( (\sigma_a) = (\sigma_a-b) \in \AOP(n,r) \). Since
  \( (\sigma_a) \) contains only one ordered partition \( \sigma_a \),
  each \(j \in [r]\) has a neighbor \( k \) in \( (\sigma_a) \). This
  shows that there is no sequence of admissible transpositions that
  turns the Bott matrix \(F(\sigma_1)\) into a block diagonal matrix
  whose diagonal blocks are Bott matrices of smaller size. Thus, by
  \Cref{prop_op1_on_BS}, \Cref{prop_op2_on_BS}, and
  \Cref{lemma_decomposable_Bott_matrix}, we obtain that
  \(F(\sigma_a)\) is indecomposable.
\end{proof}

\begin{definition}\label{def:3}
Let \(\bm i = (i_1,\dots,i_m) \in [n]^m\). For \( j\in [m-1] \), we
define \( s_j \cdot \bm i \) to be the sequence obtained from
\( \bm i \) by exchanging \( i_j \) and \( i_{j+1} \). If
\(|i_j - i_{j+1}| > 1\), the operation that sends \( \bm i \) to
\( s_j \cdot \bm i \) is called a \emph{2-move}.
\end{definition}

Recall that \(j\) and \(j+1\) are not neighbors in \(\alpha(\bm i)\)
if and only if \(|i_j - i_{j+1}| > 1\). Thus, we obtain the following
lemma.

\begin{lemma}\label{lemma_admissible_words}
  Let \(\bm i = (i_1,\dots,i_m) \in [n]^m\) and \( j\in [m-1] \). Then
  \(s_j\) is an admissible transposition of \( \alpha(\bm i)\) if and
  only if \(|i_j - i_{j+1}| > 1\). Moreover, if \(s_j\) is an
  admissible transposition of \( \alpha(\bm i)\), then
  \( s_j \cdot \alpha(\bm i) = \alpha(s_j \cdot \bm i) \).
\end{lemma}

For \( \bm i = (i_1,\dots,i_m) \in [n]^m \), the
\emph{standardization} \( \st(\bm i) \) of \( \bm i \) is the sequence
defined by
\[
  \st(\bm i) = (i_1-k+1,\dots,i_m-k+1),
\]
where \( k \) is the smallest integer in \( \bm i \). We also define
\( \iota_n(\bm i) \) to be the sequence obtained from \( \bm i \) by
replacing every integer \( i\in [n] \) with \( n+1-i \). By the
definition of the map \( \alpha \), we have
\( \alpha(\st(\bm i)) = \alpha(\bm i) \) and
\( \alpha(\iota_n(\bm i)) \sim \alpha(\bm i) \).

\begin{theorem}\label{cor_indecomposable} 
  Let \(\bm i, \bm i' \in [n]^m\). Suppose that the Bott manifolds
  \(\B_{\bm i}\) and \(\B_{\bm i'}\) are indecomposable. Then
  \(\B_{\bm i}\) and \(\B_{\bm i'}\) are isomorphic \textup{(}as toric
  varieties\textup{)} if and only if either \( \st(\bm i') \) or
  \( \st(\iota_n(\bm i')) \) is obtained from \( \st(\bm i) \) via a
  sequence of \(2\)-moves.
\end{theorem}

\begin{proof}
  The ``if'' statement follows from \Cref{thm_isom},
  \Cref{lemma_admissible_words}, and the fact that
  \( \alpha(\st(\bm i)) = \alpha(\bm i) \) and
  \( \alpha(\iota_n(\bm i)) \sim \alpha(\bm i) \).

  For the ``only if'' statement, suppose that \(\B_{\bm i}\) and
  \(\B_{\bm i'}\) are isomorphic. By \Cref{prop_indecomposable}, each
  of \(\alpha(\bm i)\) and \(\alpha(\bm i')\) consists of one ordered
  partition, i.e., we can write
  \(\alpha(\bm i) = \alpha(\st(\bm i)) = (\sigma)\) and
  \( \alpha(\bm i') = \alpha(\st(\bm i')) = (\sigma') \). By
  Theorem~\ref{thm_isom}, there exist
  \( \rho^{(0)},\rho^{(1)},\dots,\rho^{(k)}\in \AOP(n,m) \) such that
  \( \rho^{(0)}\sim (\sigma) \), \( \rho^{(k)} \sim (\sigma') \), and
  for each \( i\in [k] \), there is an admissible transposition
  \( s_{j_i} \) of \( \rho^{(i-1)} \) such that
  \( s_{j_i}\cdot \rho^{(i-1)} = \rho^{(i)} \). Since \( (\sigma) \)
  has only one ordered partition, \( \rho^{(0)}\sim (\sigma) \)
  implies that \( \rho^{(0)} = (\sigma^{(0)}) \) for some ordered
  partition \( \sigma^{(0)} \). Similarly, for each \( i\in [k] \), we
  have \( \rho^{(i)} = (\sigma^{(i)}) \) for some ordered partition
  \( \sigma^{(i)} \). Since
  \( \rho^{(0)} = (\sigma^{(0)}) \sim (\sigma) \), we have
  \( \sigma^{(0)} = \sigma \) or
  \( \sigma^{(0)} = \overline{\sigma} \). By replacing every
  \( \sigma^{(i)} \) by \( \overline{\sigma^{(i)}} \) if necessary, we
  may assume that \( \rho^{(0)} = (\sigma) = \alpha(\st(\bm i)) \). On
  the other hand, \( \rho^{(k)} \) is either \( \alpha(\st(\bm i')) \)
  or \( \alpha(\st(\iota_n(\bm i'))) \).

  Since \( s_{j_i} \) is an admissible transposition of
  \( \rho^{(i-1)} \) such that
  \( s_{j_i}\cdot \rho^{(i-1)} = \rho^{(i)} \), by
  \Cref{lemma_admissible_words}, we have
  \( \rho^{(i)} = \alpha(\bm i^{(i)}) \), where
  \( \bm i^{(0)} = \st(\bm i) \) and
  \( \bm i^{(i+1)} = s_{j_i}\cdot \bm i^{(i)} \). Thus,
  \( \rho^{(k)} = \alpha(\bm i^{(k)}) \), which is equal to either
  \( \alpha(\st(\bm i')) \) or \( \alpha(\st(\iota_n(\bm i'))) \).
  Since each of \( \bm i^{(k)} \), \( \st(\bm i') \), and
  \( \st(\iota_n(\bm i')) \) has the smallest integer \( 1 \), and
  their images under \( \alpha \) have only one ordered partition, we
  obtain that \( \bm i^{(k)} \) is equal to either \( \st(\bm i') \)
  or \( \st(\iota_n(\bm i')) \). This means that \( \st(\bm i') \) or
  \( \st(\iota_n(\bm i')) \) is obtained from \( \st(\bm i) \) via a
  sequence of \(2\)-moves.
\end{proof}

Recall from~\cite{Karu13Schubert} that a Schubert variety \(X_w\) is
toric if and only if any reduced decomposition
\(\bm i = (i_1,\dots,i_m) \) of \(w\) consists of distinct integers.
Moreover, in this case, a toric Schubert variety \(X_w\) is isomorphic
to the Bott manifold \(\B_{\bm i}\) of BS type corresponding to
\(\bm i\) (see Remark~\ref{rmk_toric_Schubert}). Applying
Theorem~\ref{cor_indecomposable} to toric Schubert varieties in type
\(A_m\) of dimension \(m\) gives another proof of the result
in~\cite[Proposition~4.1]{LMP_directed_Dynkin} for the case of type \(A\).

\begin{corollary}\label{cor_compare_with_directed_Dynkin}
    Let \(X_w\) and \(X_{w'}\) be toric Schubert varieties in type \(A_m\) of dimension \(m\). Then \(X_w\) and \(X_{w'}\) are isomorphic \textup{(}as toric varieties\textup{)} if and only if either \(w'=w\) or \(w'=w_0ww_0\). 
    Here, \(w_0\) is the longest element in the Weyl group.
\end{corollary}
\begin{proof}
  We first note that if a toric Schubert variety \(X_w\) in type
  \(A_m\) has dimension \(m\), then for any reduced decomposition
  \(\bm i = (i_1,\dots,i_m) \) of \(w\), we have
  \(\{i_1,\dots,i_m\} = [m]\). Hence, \(X_w\) is indecomposable.
  Moreover, \(w_0ww_0\) can be obtained from \(w\) by replacing every
  simple reflection \(s_i\) in a reduced decomposition with
  \(s_{m+1-i}\). Since a sequence of \(2\)-moves on a sequence does not
  change the corresponding permutation, we obtain the result by
  Theorem~\ref{cor_indecomposable}.
\end{proof}


\begin{thebibliography}{10}
	
	\bibitem{AndersonFulton24}
	David Anderson and William Fulton.
	\newblock {\em Equivariant cohomology in algebraic geometry}, volume 210 of
	{\em Cambridge Studies in Advanced Mathematics}.
	\newblock Cambridge University Press, Cambridge, 2024.
	
	\bibitem{BottSamelson58}
	Raoul Bott and Hans Samelson.
	\newblock Applications of the theory of {M}orse to symmetric spaces.
	\newblock {\em Amer. J. Math.}, 80:964--1029, 1958.
	
	\bibitem{CLMP_unique}
	Yunhyung Cho, Eunjeong Lee, Mikiya Masuda, and Seonjeong Park.
	\newblock Unique toric structure on a {F}ano {B}ott manifold.
	\newblock arXiv:2005.02740v1, 2020.
	
	\bibitem{CHJ22}
	Suyoung Choi, Taekgyu Hwang, and Hyeontae Jang.
	\newblock Strong cohomological rigidity of {B}ott manifolds.
	\newblock {\em Adv. Math.}, 473:Paper No. 110305, 16, 2025.
	
	\bibitem{ch-ma-ou17}
	Suyoung Choi, Mikiya Masuda, and Sang-il Oum.
	\newblock Classification of real {B}ott manifolds and acyclic digraphs.
	\newblock {\em Trans. Amer. Math. Soc.}, 369(4):2987--3011, 2017.
	
	\bibitem{CLS11Toric}
	David~A. Cox, John~B. Little, and Henry~K. Schenck.
	\newblock {\em Toric varieties}, volume 124 of {\em Graduate Studies in
		Mathematics}.
	\newblock American Mathematical Society, Providence, RI, 2011.
	
	\bibitem{Demazure74}
	Michel Demazure.
	\newblock D\'esingularisation des vari\'et\'es de {S}chubert
	g\'en\'eralis\'ees.
	\newblock {\em Ann. Sci. \'Ecole Norm. Sup. (4)}, 7:53--88, 1974.
	
	\bibitem{GK94Bott}
	Michael Grossberg and Yael Karshon.
	\newblock Bott towers, complete integrability, and the extended character of
	representations.
	\newblock {\em Duke Math. J.}, 76(1):23--58, 1994.
	
	\bibitem{Hansen73}
	H.~C. Hansen.
	\newblock On cycles in flag manifolds.
	\newblock {\em Math. Scand.}, 33:269--274, 1973.
	
	\bibitem{Hartshorne}
	Robin Hartshorne.
	\newblock {\em Algebraic geometry}, volume No. 52 of {\em Graduate Texts in
		Mathematics}.
	\newblock Springer-Verlag, New York-Heidelberg, 1977.
	
	\bibitem{Karu13Schubert}
	Paramasamy Karuppuchamy.
	\newblock On {S}chubert varieties.
	\newblock {\em Comm. Algebra}, 41(4):1365--1368, 2013.
	
	\bibitem{LMP_directed_Dynkin}
	Eunjeong Lee, Mikiya Masuda, and Seonjeong Park.
	\newblock {T}oric {S}chubert varieties and directed {D}ynkin diagrams.
	\newblock {\em in preparation}.
	
	\bibitem{Pasquier10}
	Boris Pasquier.
	\newblock Vanishing theorem for the cohomology of line bundles on
	{B}ott-{S}amelson varieties.
	\newblock {\em J. Algebra}, 323(10):2834--2847, 2010.
	
	\bibitem{Stanley_EC_2}
	Richard~P. Stanley.
	\newblock {\em Enumerative combinatorics. {V}ol. 2}, volume 208 of {\em
		Cambridge Studies in Advanced Mathematics}.
	\newblock Cambridge University Press, Cambridge, second edition, [2024]
	\copyright 2024.
	\newblock With an appendix by Sergey Fomin.
	
\end{thebibliography}

\end{document}